\documentclass{amsart}
\usepackage{amsfonts}
\usepackage{amsmath}
\usepackage{amssymb}
\usepackage{mathrsfs}
\usepackage[all]{xy}
\usepackage{color}

\newtheorem{theorem}{Theorem}[section]
\newtheorem{lemma}[theorem]{Lemma}
\newtheorem{corollary}[theorem]{Corollary}
\theoremstyle{definition}
\newtheorem{definition}[theorem]{Definition}

\newtheorem{proposition}[theorem]{Proposition}
\theoremstyle{remark}
\newtheorem{remark}[theorem]{Remark}
\numberwithin{equation}{section}

\def\M{\mathcal M}

\def\N{\mathcal N}
\keywords{$C$-numerical radius, von Neumann algebras, factors, finite factors, unitary conjugates, weakly unitarily invariant norm, Aluthge transform, the $\lambda$-Aluthge transform}
\subjclass[2000]{Primary 47A12, Secondary 46L10}
\begin{document}
\title[]{A note on the $C$-numerical radius and the $\lambda$-Aluthge transform in finite factors}
\author{Xiaoyan Zhou}
\author{Junsheng Fang}
\author{Shilin Wen}
\address{Xiaoyan Zhou}\address{School of Mathematical Sciences, Dalian University of Technology. Dalian~{116024}. China}
\email{doctoryan@mail.dlut.edu.cn}

\address{Junsheng Fang}\address{School of Mathematical Sciences, Dalian University of Technology. Dalian~{116024}. China}
\email{junshengfang@hotmail.com}

\address{Shilin Wen}\address{School of Mathematical Sciences, Dalian University of Technology. Dalian~{116024}. China}
\email{shilinwen127@gmail.com}

\begin{abstract}
We prove that for any two elements $A$, $B$ in a factor $\M$, if $B$ commutes with all the unitary conjugates of $A$, then either $A$ or $B$ is in $\mathbb{C}I$. Then we obtain an equivalent condition for the situation that the $C$-numerical radius $\omega_{C}(\cdot)$ is a weakly unitarily invariant norm on finite factors and we also prove some inequalities on the $C$-numerical radius on finite factors. As an application, we show that for an invertible operator $T$ in a finite factor $\M$,  $f(\bigtriangleup_{\lambda}(T))$ is in the weak operator closure of the set $\{\sum_{i=1}^{n}z_{i}U_{i}f(T)U_{i}^{*}|~n\in \mathbb{N},(U_{i})_{1\leq i\leq n}\in \mathscr{U}(\M),\sum_{i=1}^{n}|z_{i}|\leq 1\}$, where $f$ is a polynomial, $\bigtriangleup_{\lambda}(T)$ is the $\lambda$-Aluthge transform of $T$ and $0\leq\lambda \leq 1$.
\end{abstract}
\maketitle
\section{notation and introduction}
Denote by $B(\mathscr{H})$ the set of bounded linear operators on a Hilbert space $\mathscr{H}$ and $M_{n}(\mathbb{C})$ the self-adjoint algebra of the $n\times n$ matrices. A von Neumann algebra $\M$ on $\mathscr{H}$ is a unital weak operator closed $*$-algebra. A von Neumann algebra $\M$ is said to be a factor if $\M\cap \M'=\mathbb{C}I$, where $I$ is the identity of $\M$.  A von Neumann algebra $\M$ is finite if it has a faithful normal tracial state. If $\M$ is a finite factor with a faithful normal trace $\tau$, denote by $\|\cdot\|_{1}$ the norm on $\M$ to be $\tau(|\cdot|)$. Then denote by $L^{1}(\M,\tau)$ the completion of $\M$ with respect to $\|\cdot\|_{1}$ norm. Also to each normal linear functional $f$ on $\M$ corresponds a unique element $X\in L^{1}(\M,\tau)$ such that $f(\cdot)=\tau(X\cdot)$. Denote by $\mathscr{U}(\M)$ the set of all the unitary operators in a von Neumann algebra $\M$.

Let $tr$ be the normalized trace of $M_{n}(\mathbb{C})$. Given a matrix $C\in M_{n}(\mathbb{C})$ and set $$\omega_{C}(A)=\max\limits_{U\in \mathscr{U}( M_{n}(\mathbb{C}))}|tr (CUAU^{*})|.$$ Then $\omega_{C}(A)$ is called the \emph{$C$-numerical radius} of $A$. We say a norm $|||\cdot|||$ on $M_{n}(\mathbb{C})$ \emph{weakly unitarily invariant} if $|||A|||=|||UAU^{*}|||$ for all $A\in M_{n}(\mathbb{C}), U\in\mathscr{U}( M_{n}(\mathbb{C}))$. Note that for every $C\in M_{n}(\mathbb{C})$, the $C$-numerical radius $\omega_{C}$ is a weakly unitarily invariant seminorm on $M_{n}(\mathbb{C})$. It is a norm on
$M_{n}(\mathbb{C})$ if and only if $C$ is not a scalar and has nonzero trace. The family $\omega_{C}$ of $C$-numerical radius, where $C$ is not a scalar and has nonzero trace, plays a role analogous to that of Ky Fan norms in the family of unitarily invariant norm \cite[Theorem IV.4.7]{book}. A norm $|||\cdot|||$ on $M_{n}(\mathbb{C})$ is called \emph{a unitarily invariant norm} if $|||A|||=|||UAV^{*}|||$ for all $A\in M_{n}(\mathbb{C}), U,V\in\mathscr{U}( M_{n}(\mathbb{C}))$. The concept of unitarily invariant norms was introduced by von Neumann \cite{VON} for the purpose of metrizing matrix spaces. Von Neumann and his associates established that the class of unitarily invariant norms of $n\times n$ complex matrices coincides with the class of symmetric gauge function of their $s$-numbers. These norms have now been variously generalized and utilized in many contexts. For historical perspectives and surveys, we refer the reader to (\cite{book},\cite{K},\cite{IC},\cite{R},\cite{RS},\cite{B} and etc).

Let $T\in B(\mathscr{H})$ and let $T=U|T|$ be its polar decomposition. The Aluthge transform of $T$ is the operator $\bigtriangleup(T)=|T|^{\frac{1}{2}}U|T|^{\frac{1}{2}}$. This was first studied in \cite{AA} and has received much attention in recent years. One reason the Aluthge transform is interesting is in relation to the invariant subspace problem. Jung, Ko and Pearcy prove in \cite{JKP} that $T$ has a nontrivial invariant subspace if and only if $\bigtriangleup(T)$ does. They also note that when $T$ is quasiaffinity, then $T$ has a nontrivial hyperinvariant subspace if and only if $\bigtriangleup(T)$ does. A quasiaffinity is an operator with zero kernel and dense range. The invariant and hyperinvariant subspace problems are interesting only for quasiaffinities. Clearly, the spectrum of
$\bigtriangleup(T)$ equals that of $T$. Jung, Ko and Percy in \cite{JKP} proved that other spectral data are also preserved by the Aluthge transform. Dykema and Schultz in \cite{DS} proved the Brown measures are unchanged by the Aluthge transform.

Another reason is related with iterated Aluthge transform. Let
$\bigtriangleup^{0}(T)=T$ and $\bigtriangleup^{n}(T)=\bigtriangleup(\bigtriangleup^{n-1}(T))$
for every $n\in \mathbb{N}$. It was conjectured in \cite{JKP} that the sequence $\{\bigtriangleup^{n}(T)\}_{n\in\mathbb{N}}$ converges in the norm topology. For more surveys, we refer the reader to (\cite{AA},\cite{APS},\cite{DS},\cite{JKP},\cite{Okubo2},\cite{Okubo} and etc).

The $\lambda$-Aluthge transform of $T$ is defined in \cite{Okubo2} by $\bigtriangleup_{\lambda}(T)=|T|^{\lambda}U|T|^{1-\lambda} $, $0\leq\lambda \leq1$. In particular, for $\lambda=\frac{1}{2}$, $\bigtriangleup_{\frac{1}{2}}(T)$ is just the Aluthge transform $\bigtriangleup(T)$. Okubo in \cite{Okubo2} proved that for an invertible operator $T\in B(\mathscr{H})$, $\|f(\bigtriangleup_{\lambda}(T))\|\leq\|f(T)\|$ for any polynomial $f$ and $\|\cdot\|$ a weakly unitarily invariant norm. Fore more results on the $\lambda$-Aluthge transform, we refer the reader to (\cite{CM},\cite{Okubo2},\cite{Okubo} and etc)

This paper is organized as follows.

The key motivation for studying the $C$-numerical radius $\omega_{C}$ on finite factors stems from the fact that for the finite dimensional case, i.e., $M_{n}(\mathbb{C})$, it has a relation with weakly unitarily invariant norms on $M_{n}(\mathbb{C})$. So in section 2, we use some knowledge on dual norms to show that relation.

In section 3, We first prove that if $\M$ is a factor, then for any non-trivial projection $P$ in $\M$, all the unitary conjugates of $P$ generate the whole von Neumann algebra $\M$ (see Lemma \ref{5}).
Then using this lemma we prove a technical result in this paper.

\begin{theorem}[see Theorem \ref{1}]
Let $\M$ be a factor and $A,B\in \M$. If $UAU^{*}B=BUAU^{*}$ holds for every $U\in \mathscr{U}(\M)$, then either $A$ or $B$ is in $\mathbb{C}I$.
\end{theorem}


We define the $C$-numerical radius on finite factors.
\begin{definition}Let $\M$ be a finite factor with a faithful normal trace $\tau$ and for $A,C\in\M$, the $C$-numerical radius of $A$ is defined as $$\omega_{C}(A)=\sup\limits_{U\in \mathscr{U}(\M)}|\tau (CUAU^{*})|.$$\end{definition}

Observe that the $C$-numerical radius of $A$ is a weakly unitarily invariant seminorm on $\M$.

In section 4, as one application of Theorem 1.1, we prove the following corollary.
\begin{corollary}[see Corollary \ref{2}]
Let $\M$ be a finite factor with a faithful normal trace $\tau$. The $C$-numerical radius $\omega_{C}$ is a norm on $\M$ if and only if
\begin{enumerate}
\item $C$ is not a scalar multiple of $I$ and;
\item $\tau(C)\neq0$.
\end{enumerate}
\end{corollary}

We also prove some inequalities for the $C$-numerical radius $\omega_{C}$ on finite factors (see Theorem \ref{3}).

In section 5, we discuss some properties of the $\lambda$-Aluthge transform of an invertible operator in a finite factor. Using three line theorem and some results in section 4, we obtain the following result.
\begin{proposition}[see Proposition \ref{5.3}]
Let $M$ be a finite factor with a faithful normal trace $\tau$. Assume $T\in \M$ is an invertible operator with polar decomposition $T=U|T|$ and $f$ is a polynomial, then for  $0\leq\lambda \leq 1$, $f(|T|^{\lambda}U|T|^{1-\lambda})$ is in the weak operator closure of the set $\{\sum_{i=1}^{n}z_{i}U_{i}f(T)U_{i}^{*}|~n\in \mathbb{N},(U_{i})_{1\leq i\leq n}\in \mathscr{U}(\M),\sum_{i=1}^{n}|z_{i}|\leq 1\}$.
\end{proposition}

In this paper, we assume all the factors have separable predual.

\section{relation between weakly unitarily invariant norms and the $C$-numerical radius $\omega_{C}$ on $M_{n}(\mathbb{C})$}
In this section, a finite von Neumann algebra $(\M,\tau)$ means a finite von Neumann algebra $\M$ with a faithful normal tracial state $\tau$. Recall the definition and some properties of dual norms in \cite{fang}.

Let $|||\cdot|||$ be a norm on a finite von Neumann algebra $(\M,\tau)$. For $T\in\M$, define $$|||T|||^{\sharp}_{\M}=\sup\{|\tau(TX)|:X\in\M,|||X|||\leq 1\}.$$ When no confusion arises, we write $|||\cdot|||^{\sharp}$ instead of $|||\cdot|||^{\sharp}_{\M}$.

\begin{lemma}[\cite{fang}]$|||\cdot|||^{\sharp}$ is a norm on $(\M,\tau)$.
\end{lemma}

\begin{definition}[\cite{fang}]$|||\cdot|||^{\sharp}$ is called the \emph{dual norm} of $|||\cdot|||$ on $\M$ with respect to $\tau$.
\end{definition}

\begin{definition} A norm $|||\cdot|||$ on $(\M,\tau)$  is weakly unitarily invariant if $|||UTU^{*}|||=|||T|||$  for all $T\in \M$ and $U\in \mathscr{U}(\M)$.
\end{definition}

\begin{lemma}[\cite{fang}]\label{7} If $|||\cdot|||$ is a norm on $(M_{n}(\mathbb{C}),tr)$ and $|||\cdot|||^{\sharp}$ is the dual norm with respect to $tr$, then $|||\cdot|||=|||\cdot|||^{\sharp\sharp}$.
\end{lemma}

\begin{lemma}\label{6} If $|||\cdot|||$ is a weakly unitarily invariant norm on a finite von Neumann algebra $(\M,\tau)$, then $|||\cdot|||^{\sharp}$ is also a weakly unitarily invariant norm on $(\M,\tau)$.
\end{lemma}
\begin{proof}
Let $U\in \mathscr{U}(\M)$. Then $|||UTU^{*}|||^{\sharp}=\sup\{|\tau(UTU^{*}X)|:X\in\M,|||X|||\leq 1\}=\sup\{|\tau(TU^{*}XU)|:X\in\M,|||U^{*}XU|||\leq 1\}=|||T|||^{\sharp}$.
\end{proof}

We now proceed to the relation between weakly unitarily invariant norms and the $C$-numerical radius on $(M_{n}(\mathbb{C}),tr)$.
\begin{proposition}\label{8}
If $|||\cdot|||$ is a weakly unitarily invariant norm on $(M_{n}(\mathbb{C}),tr)$, then $|||T|||=\sup\limits_{|||X|||^{\sharp}\leq1}\omega_{X}(T)$.
\end{proposition}
\begin{proof}
For $T\in (M_{n}(\mathbb{C}),tr)$, by Lemma \ref{6}, Lemma \ref{7} and the definition of dual norm, we have
\begin{align*}
|||T|||=|||T|||^{\sharp\sharp}&=\sup\limits_{U\in \mathscr{U}(\M)}|||UTU^{*}|||^{\sharp\sharp}\\
&=\sup\limits_{U\in \mathscr{U}(\M)}\sup\limits_{|||X|||^{\sharp}\leq1}\{|\tau(TUXU^{*})|,X\in M_{n}(\mathbb{C})\}\\
&=\sup\limits_{|||X|||^{\sharp}\leq1}\sup\limits_{U\in \mathscr{U}(\M)}\{|\tau(TUXU^{*})|,X\in M_{n}(\mathbb{C})\}\\
&=\sup\limits_{|||X|||^{\sharp}\leq1}\omega_{X}(T).
\end{align*}.
\end{proof}
Note that when proving Proposition \ref{8}, we use Lemma \ref{7} \cite[Lemma 6.18]{fang}, so we may ask whether this result can be generalized to finite factors.

\section{A result on factors}
In this section, we show a technical result (Theorem \ref{1}), which is the most difficult part of this paper.
To prove that result, we first need the following lemma.
\begin{lemma}\label{5}
Let $\M$ be a factor and $P$ be a non-trivial projection in $\M$. Then the von Neumann algebra generated by $\{UPU^{*}:U\in \mathscr{U}(\M)\}$ is $\M$.
\end{lemma}
\begin{proof}
We divide the proof into four cases according to the the type of $\M$.

(i) For the case $\M=B(\mathscr{H})$, where $dim(\mathscr{H})\leq\infty$.

Take two projections $P_{0}\leq P$ and $P_{1}\leq 1-P$ with $dim(P_{i}(H))=1$ for $i=0,1$ and write $Q=P-P_{0}+P_{1}$, then $P_{0}=P(1-Q)$ and we can find some unitary operator $V\in\mathscr{U}(\M)$ such that $VPV^{*}=Q$, since $P$ and $Q$ are equivalent. Then we have $\{UP_{0}U^{*}:U\in \mathscr{U}(\M)\}''\subseteq \{UPU^{*}:U\in \mathscr{U}(\M)\}''$. Note that the von Neumann algebra generated by $\{UP_{0}U^{*}:U\in \mathscr{U}(\M)\}$ is $\M$. Hence we prove our result.

(ii) For the case $\M$ is a $II_{1}$ factor with a faithful normal tracial state $\tau$.

Write $\tau(P)=\lambda\in(0,1)$ and we may assume $\lambda\leq\frac{1}{2}$. Then for any $0<t\leq\lambda$, we can find two projections $P_{t}\leq P$ and $F_{t}\leq 1-P$ with $\tau(P_{t})=\tau(F_{t})=t$.
Write $Q_{t}=P-P_{t}+F_{t}$, then $P_{t}=P(1-Q_{t})$. Again we can find some unitary operator $V\in\mathscr{U}(\M)$ such that $VPV^{*}=Q_{t}$. Hence
$\{UP_{t}U^{*}:\tau(P_{t})=t\in(0,\lambda],P_{t}\leq P, U\in \mathscr{U}(\M)\}''\subseteq \{UPU^{*}:U\in \mathscr{U}(\M)\}''$. Note that the von Neumann algebra generated by $\{UP_{t}U^{*}:\tau(P_{t})=t\in(0,\lambda],P_{t}\leq P, U\in \mathscr{U}(\M)\}$ is the whole $\M$. Then we have our result.

(iii) For the case $\M$ is a $II_{\infty}$ factor with a faithful normal tracial weight $Tr$.

Write $Tr(P)=\lambda\in(0,\infty]$ and we may assume $Tr(1-P)\geq Tr(P)$. Then using the same trick in case (ii), we prove our result.


(iv) For the case $\M$ is a type $III$ factor.

This case is trivial, since all the non-trivial projections in a type $III$ factor are equivalent.
\end{proof}


Our main theorem is the following.
\begin{theorem}\label{1}
Let $\M$ be a factor and $A,B\in \M$. If $UAU^{*}B=BUAU^{*}$ holds for any $U\in \mathscr{U}(\M)$, then either $A$ or $B$ is in $\mathbb{C}I$.
\end{theorem}
\begin{proof}
Let $P$ be a projection in $\M$, then we can write $A$ and $B$ in the matrix form
$A=\begin{pmatrix}
A_{11}&A_{12}\\
A_{21}&A_{22}\\
\end{pmatrix},B=\begin{pmatrix}
B_{11}&B_{12}\\
B_{21}&B_{22}\\
\end{pmatrix},$ where $A_{11},B_{11}\in P\M P$, $A_{12},B_{12}\in P\M P^{\bot}$, $A_{21},B_{21}\in P^{\bot}\M P$, $A_{22},B_{22}\in P^{\bot}\M P^{\bot}.$

Let $\theta\in[0,2\pi]$, $U=\begin{pmatrix}
e^{i\theta}P_{n}&0\\
0&P_{n}^{\perp}\\
\end{pmatrix}$, it is clear that $U$ is a unitary operator. Then we have $UAU^{*}=\begin{pmatrix}
A_{11}&e^{i\theta}A_{12}\\
e^{-i\theta}A_{21}&A_{22}\\
\end{pmatrix}$,
$$
UAU^{*}B=\begin{pmatrix}
A_{11}&e^{i\theta}A_{12}\\
e^{-i\theta}A_{21}&A_{22}\\
\end{pmatrix}\begin{pmatrix}
B_{11}&B_{12}\\
B_{21}&B_{22}\\
\end{pmatrix}=\begin{pmatrix}
A_{11}B_{11}+e^{i\theta}A_{12}B_{21}&\ast\\
\ast&\ast\\
\end{pmatrix},
$$and
$$
BUAU^{*}=\begin{pmatrix}
B_{11}&B_{12}\\
B_{21}&B_{22}\\
\end{pmatrix}\begin{pmatrix}
A_{11}&e^{i\theta}A_{12}\\
e^{-i\theta}A_{21}&A_{22}\\
\end{pmatrix}=\begin{pmatrix}
B_{11}A_{11}+e^{-i\theta}B_{12}A_{21}&\ast\\
\ast&\ast\\
\end{pmatrix}.
$$
It follows that
\begin{equation}\label{eq:1}
A_{11}B_{11}-B_{11}A_{11}+e^{i\theta}A_{12}B_{21}-e^{-i\theta}B_{12}A_{21}=0
\end{equation} since $UAU^{*}B=BUAU^{*}$.
Note that \eqref{eq:1} holds for any $\theta\in[0,2\pi]$, a not difficult calculation implies
\begin{equation}\label{eq:2}
A_{11}B_{11}=B_{11}A_{11}, A_{12}B_{21}=B_{12}A_{21}=0.
\end{equation}

Observe that for any $U,V\in \mathscr{U}(\M)$, $UVAV^{*}U^{*}B=BUVAV^{*}U^{*}$ still holds, in particular, we can choose $V=\begin{pmatrix}
V_{1}&0\\
0&P^{\perp}\\
\end{pmatrix}$, where $V_{1}\in \mathscr{U}(P\M P)$, then
\begin{equation}\label{eq:3}
V_{1}A_{11}V^{*}_{1}B_{11}=B_{11}V_{1}A_{11}V^{*}_{1}.
\end{equation}

(i) For the case $\M=B(\mathscr{H})$, where $dim(\mathscr{H})=\infty$.

For $n\in \mathbb{N}$, let $P_{n}$ be a projection of dimension $n$ and $P_{n}\leq P_{n+1}$.

By a result of finite dimension case, i.e., if $A,B\in M_{n}(\mathbb{C})$ and $UAU^{*}B=BUAU^{*}$ holds for any $U\in \mathscr{U}(M_{n}(\mathbb{C}))$, then either $A$ or $B$ is in $\mathbb{C}I_{n}$, where $I_{n}$ is the identity of $M_{n}(\mathbb{C})$(cf. proof of \cite[Proposition IV.4.4]{book}). Then by \eqref{eq:3}, we have either $A_{11}$ or $B_{11}$ is in $\mathbb{C}I_{n}$, i.e., $P_{n}AP_{n}$ or $P_{n}BP_{n}$ is in $\mathbb{C}I_{n}$, for any $n\in \mathbb{N}$. Assume $P_{n}AP_{n}$ is in $\mathbb{C}I_{n}$, while $P_{n}BP_{n}$ not. For $m>n$, if $P_{m}AP_{m}$ isn't in $\mathbb{C}I_{m}$, while $P_{m}BP_{m}$ is in $\mathbb{C}I_{m}$, that would contradict the assumption $P_{n}BP_{n}$ isn't in $\mathbb{C}I_{n}$. Hence we have for all $n\in \mathbb{N}$, $P_{n}AP_{n}$ is in $\mathbb{C}I_{n}$, which implies
$A$ is in $\mathbb{C}I$.

(ii) For the case $\M$ is a $II_{1}$ factor with trace $\tau$ or a type $III$ factor.

If $\M$ is a  $II_{1}$ factor, then assume $\tau(P)=\frac{1}{2}$. Otherwise if $\M$ is a type $III$ factor, then assume $P\neq 0$ or $P\neq 1$. Then we have $\M\cong M_{2}(\mathbb{C})\otimes P\M P$ and we can write $A,B$ in the matrix form
$$A=\begin{pmatrix}
A_{11}&A_{12}\\
A_{21}&A_{22}\\
\end{pmatrix},B=\begin{pmatrix}
B_{11}&B_{12}\\
B_{21}&B_{22}\\
\end{pmatrix}, A_{ij},B_{ij}\in P\M P,\mbox{~for~}1\leq i,j\leq 2.$$

Let $V_{1},V_{2}\in \mathscr{U}(P\M P)$ and put
$V=\begin{pmatrix}
V_{1}&0\\
0&V_{2}\\
\end{pmatrix}$, then we have $$VAV^{*}=\begin{pmatrix}
V_{1}A_{11}V^{*}_{1}&V_{1}A_{12}V^{*}_{2}\\
V_{2}A_{21}V^{*}_{1}&V_{2}A_{22}V^{*}_{2}\\
\end{pmatrix}.$$
It follows that $V_{1}A_{12}V^{*}_{2}B_{21}=0$, since $UVAV^{*}U^{*}B=BUVAV^{*}U^{*}$ for any $U,V\in \mathscr{U}(\M)$ and \eqref{eq:2}.
If $A_{12}\neq0$, then $A_{12}V^{*}_{2}B_{21}=B^{*}_{21}V_{2}A^{*}_{12}=0$ for all unitary operator $V_{2}\in\mathscr{U}(P\M P)$, which implies $B_{21}=0$.
Moreover, put $V'=\begin{pmatrix}
0&V_{1}\\
V_{2}&0\\
\end{pmatrix}$, then $$V'AV'^{*}=\begin{pmatrix}
V_{1}A_{22}V^{*}_{1}&V_{1}A_{21}V^{*}_{2}\\
V_{2}A_{12}V^{*}_{1}&V_{2}A_{11}V^{*}_{2}\\
\end{pmatrix}.$$ Using the same trick as above, we obtain that if $A_{12}\neq0$, then $B_{12}=0$. Thus we have if $A_{12}\neq0$, then $B_{21}=B_{12}=0$. Similarly, we would have if $A_{21}\neq0$, then $B_{21}=B_{12}=0$.

Observe that if we replace $A$ with $UAU^{*}$ for every $U\in \mathscr{U}(\M)$ and replace $B$ with $VBV^{*}$ for every $V\in \mathscr{U}(\M)$, then the above fact still holds.

Then we can argue as follows.

Assume that $A\notin\mathbb{C}I$, we try to show $B\in \mathbb{C}I$.

Case 1: If there exists $U\in \mathscr{U}(\M)$ such that $(UAU^{*})_{12}$ or $(UAU^{*})_{21}$ is non-zero, then from above, we know that $(VBV^{*})_{12}=(VBV^{*})_{21}=0$ for every $V\in \mathscr{U}(\M)$. Hence $VBV^{*}P=PVBV^{*}$ for every $V\in \mathscr{U}(\M)$. Then apply Lemma \ref{5} to get
$B\in \mathbb{C}I$.

Case 2: If for every $U\in \mathscr{U}(\M)$, $(UAU^{*})_{12}=(UAU^{*})_{21}=0$. Then $UAU^{*}P=PUAU^{*}$ for every $U\in \mathscr{U}(\M)$. Again using Lemma \ref{5}, we have $A\in\mathbb{C}I$, which is a contradiction. Hence this case actually does not appear under the assumption that $A\notin\mathbb{C}I$.

(iii) For the case $\M$ is a $II_{\infty}$ factor.

Note that $\M=B(\mathscr{H})\otimes \N$, where $\N$ is a $II_{1}$ factor. For any $n\in \mathbb{N}$, let $P'_{n}$ be a projection of dimension $n$ in $B(\mathscr{H})$, $I'$ be the identity of $\N$ and $P_{n}=P'_{n}\otimes I'$, then $P_{n}\M P_{n}$ is a type $II_{1}$ factor. Hence using the same trick in case (i) and the result in case (ii), our result follows.
\end{proof}

\section{ The $C$-numerical radius $\omega_{C}$ on finite factors}
In this section, we show some applications of Theorem \ref{1} and discuss some properties of the $C$-numerical radius $\omega_{C}$ on finite factors.

We use Theorem \ref{1} and the same technique in \cite[Proposition IV.4.4]{book}, to prove our next corollary, for reader's convenience, we write the proof below.
\begin{corollary}\label{2}
Let $\M$ be a finite factor with trace $\tau$. The $C$-numerical radius $\omega_{C}$ is a weakly unitarily invariant norm on $\M$ if and only if
\begin{enumerate}
\item $C$ is not a scalar multiple of $I$ and;
\item $\tau(C)\neq0$.
\end{enumerate}
\end{corollary}
\begin{proof}
If $C=\lambda I$ for any $\lambda\in\mathbb{C}$, then $\omega_{C}(A)=|\lambda||\tau(A)|$, and this is zero if $\tau(A)=0$, which means $\omega_{C}$
can't be a norm on $\M$. If $\tau(C)=0$, then $\omega_{C}(I)=0$. Again $\omega_{C}$ is not a norm.

Conversely, suppose $\omega_{C}$ is not a norm on $\M$ and $\omega_{C}(A)=0$. If $A=\lambda I$ for any $\lambda\in\mathbb{C}$, this would mean that
$\tau(C)=0$. So, if $\tau(C)\neq 0$, then $A\notin\mathbb{C}I$. We claim that $C\in \mathbb{C}I$. Since $e^{itK}$ is in $\mathscr{U}(\M)$ for all $t\in\mathbb{R}$ and $K=K^{*}\in\M$, the condition $\omega_{C}(A)=0$ implies in particular that $\tau(Ce^{itK}Ae^{-itK})=0$ if $t\in\mathbb{R}$ and $K=K^{*}\in\M$. Differentiating this relation at $t=0$, one gets $\tau((AC-CA)K)=0$ for all $K=K^{*}\in\M$. Hence we obtain that $\tau((AC-CA)T)=0$ for
all $T\in\M$. Hence $AC=CA$. Note that $\omega_{C}(A)=\omega_{C}(UAU^{*})$ for all $U\in \mathscr{U}(\M)$, so that $UAU^{*}C=CUAU^{*}$ for all $U\in \mathscr{U}(\M)$. Hence the result $C$ is in $\mathbb{C}I$ follows from Theorem \ref{1}.
\end{proof}

Observe that for $A,C\in\M$, by the definition of the $C$-numerical radius $\omega_{C}$, we have $\omega_{C}(A)=\omega_{A}(C)$ and $\omega_{C}(\cdot)$ is normal on $\M$.
\begin{theorem}\label{3}
Let $\M$ be a finite factor with a faithful normal trace $\tau$. For $A,B\in \M$, the following conditions are equivalent.
\begin{enumerate}
\item $\omega_{C}(A)\leq \omega_{C}(B)$ for all operators $C\in \M$ that are not scalars and have nonzero trace;
\item $\omega_{C}(A)\leq \omega_{C}(B)$ for all operators $C\in \M$;
\item Let $K=\{\sum_{i=1}^{n}z_{i}U_{i}BU_{i}^{*}|~n\in \mathbb{N},(U_{i})_{1\leq i\leq n}\in \mathscr{U}(\M),\sum_{i=1}^{n}|z_{i}|\leq 1\}$ and $\Gamma$ be the weak operator closure of $K$. Then $A\in \Gamma$.
\end{enumerate}
\end{theorem}
\begin{proof}
$(1)\Rightarrow(2)$.
Assume $C\in \M$ and $\tau(C)=0$. Put $C_{n}=C+\frac{1}{n}$, then $\tau(C_{n})=\frac{1}{n}$ and $\|C_{n}-C\|\rightarrow 0$.
Moreover, we have \begin{align*}
|\omega_{A}(C_{n})-\omega_{A}(C)|&\leq\sup\limits_{U\in \mathscr{U}(\M)}|\tau(AU(C_{n}-C)U^{*})| \\
&=\sup\limits_{U\in \mathscr{U}(\M)}\frac{1}{n}|\tau(A)|\\
&\rightarrow 0.
\end{align*}

Similarly, we would have $\omega_{B}(C_{n})\rightarrow\omega_{B}(C)$. Note that $\omega_{A}(C_{n})\leq \omega_{B}(C_{n})$, then we have $\omega_{A}(C)\leq \omega_{B}(C)$.

Let $P\in \M$ be a projection with trace not equal to 0 or 1. Let $C_{n}=P+(1-\frac{1}{n})(1-P)$, then $C_{n}$ is not a scalar, $\tau(C_{n})\neq0$ and
$\|C_{n}-1\|\rightarrow 0$.
Hence we have $\omega_{A}(C_{n})\leq \omega_{B}(C_{n})$ and for any operator $T\in\M$,
\begin{align*}
|\omega_{T}(C_{n})-\omega_{T}(I)|&\leq|\omega_{T}(C_{n}-I)|\\
&=\sup\limits_{U\in \mathscr{U}(\M)}|\tau(TU(C_{n}-I)U^{*})|\\
&\leq \|C_{n}-1\|\|T\|_{1}\\
&\rightarrow 0.
\end{align*}
It follows that $\omega_{A}(I)\leq \omega_{B}(I)$.

$(2)\Rightarrow(3)$.
Assume $A\notin \Gamma$, then there exists a linear normal functional $f$ on $\M$ and $a>b$, such that
$\textmd{Re~}f(A)\geq a> b \geq \textmd{Re~}f(D),~\forall~ D\in\Gamma.$
Since  $f$ is a normal linear functional on $\M$, there exists a $C\in L^{1}(\M,\tau)$ such that $f(T)=\tau(CT)$ for all $T\in \M$.

Observe that
$\omega_{C}(A)=\sup\limits_{U\in \mathscr{U}(\M)}|\tau(CUAU^{*})|\geq|\tau(CA)|=|f(A)|$ and
$$\textmd{Re~} f(A)>\sup\limits_{D\in \Gamma}\textmd{Re~}f(D)\geq\sup\limits_{\theta,U}\textmd{Re~}f(e^{i\theta}UBU^{*})=\sup\limits_{U\in\mathscr{U}(\M)}|f(UBU^{*})|=\omega_{C}(B).$$
Let $C=V|C|$ be the polar decomposition of $C$ in $L^{1}(\M,\tau)$ and $H_{n}=\chi_{[0,n]}(|C|)|C|$, then $\|H_{n}-|C|\|_{1}\rightarrow 0$.
Put $C_{n}=VH_{n}$. Then we have

\begin{align*}
|\omega_{C_{n}}(A)-\omega_{C}(A)|&=|\omega_{A}(C_{n})-\omega_{A}(C)|\\
&\leq\sup\limits_{U\in \mathscr{U}(\M)}|\tau((C_{n}-C)UAU^{*})|\\
&\leq\|C_{n}-C\|_{1}\|A\|\\
&\rightarrow 0.
\end{align*}
Similarly, $|\omega_{C_{n}}(B)-\omega_{C}(B)|\rightarrow 0$. Hence there exists $m\in \mathbb{N}$ such that $\omega_{C_{m}}(A)>\omega_{C_{m}}(B)$, which contradicts to (3) since $C_{m}\in \M$.

$(3)\Rightarrow(1)$.

For all operators $C\in \M$ that are not scalars and have nonzero trace, by Corollary \ref{2}, we obtain that $\omega_{C}$ is a norm, hence
$\omega_{C}(T)\leq \omega_{C}(B)$ for all $T\in K$. Hence our result follows since $\omega_{C}$ is normal.
\end{proof}

 \begin{remark}If $|||\cdot|||$ is a weakly unitarily invariant norm on $(M_{n}(\mathbb{C}),tr)$. By Theorem \ref{3} and Proposition \ref{8}, we have \cite[Theorem IV.4.7]{book}.
 \end{remark}

\section{$\lambda$-Aluthge transform of an invertible operator in a finite factor}
Let $T\in B(\mathscr{H})$ and let $T=U|T|$ be its polar decomposition. The Aluthge transform of $T$ is the operator $\bigtriangleup(T)=|T|^{\frac{1}{2}}U|T|^{\frac{1}{2}}$.
The $\lambda$-Aluthge transform of $T$ is defined by $\bigtriangleup_{\lambda}(T)=|T|^{\lambda}U|T|^{1-\lambda} $, $0\leq\lambda \leq 1$.

In this section, we show some results on the $\lambda$-Aluthge transform of an invertible operator in a finite factor.

 For the infinite factor $B(\mathscr{H})$, Okubo in \cite{Okubo2} proved that if $T\in B(\mathscr{H})$ is an invertible operator, then for any polynomial $f$, $0\leq\lambda \leq 1$ and $\|\cdot\|$ a weakly unitarily invariant norm, we have $\|f(\bigtriangleup_{\lambda}(T))\|\leq\|f(T)\|$. Note that the $C$-numerical radius is a weakly unitarily invariant seminorm on a finite factor $\M$ and we have already given an equivalent condition for the situation that when this seminorm is a norm in section 4.

The idea of proving the following theorem comes from \cite{Okubo2}.
\begin{theorem}\label{4}
Let $\M$ be a finite factor with a faithful normal trace $\tau$, $T\in \M$ be an invertible operator with polar decomposition $T=U|T|$ and $B\in \M$ commute with T. Let $\omega_{C}(\cdot)$ be the $C$-numerical radius on $\M$. Then
\begin{equation}\omega_{C}(|T|^{\lambda}BU|T|^{1-\lambda})\leq \omega_{C}(BT),\mbox{~~for~~}0\leq\lambda\leq1.\end{equation}
\end{theorem}
\begin{proof}
On the strip $\{z:-\frac{1}{2}\leq Re(z) \leq \frac{1}{2}\}$, consider the operator-valued function $\phi(z)$ defined by
$$\phi(z)=|T|^{\frac{1}{2}-z}BU|T|^{\frac{1}{2}+z}.$$
It is clear that $\phi(z)$ is analytic in the interior of the strip.

For any $U\in \mathscr{U}(\M)$, define $f_{U}(z)=\tau(CU\phi(z)U^{*})$. Then $f_{U}(z)$ is uniformly bounded on the strip \color{black}and analytic since $\tau$ is linear and $\phi(z)$ is analytic. Applying three line theorem (see \cite[pp. 136-137]{IC}) to $f_{U}(z)$ we would obtain that the function
$$x\mapsto Log \sup\limits_{ y\in \mathbb{R}}|f_{U}(x+iy)|\mbox{~is~a~convex~function~on~}[-\frac{1}{2},\frac{1}{2}].$$
Put $F_{U}(x)=Log \sup\limits_{ y\in \mathbb{R}}|f_{U}(x+iy)|$, then for $-\frac{1}{2}\leq x \leq \frac{1}{2},$
$$F_{U}(x)\leq F_{U}(-\frac{1}{2})(x+\frac{1}{2})+ F_{U}(\frac{1}{2})(\frac{1}{2}-x),$$
so that

\begin{equation}
\sup\limits_{U\in \mathscr{U}(\M)}F_{U}(x)\leq\sup\limits_{U\in \mathscr{U}(\M)}F_{U}(-\frac{1}{2})(x+\frac{1}{2})+ \sup\limits_{U\in \mathscr{U}(\M)}F_{U}(\frac{1}{2})(\frac{1}{2}-x).\end{equation}

For $-\infty <y <\infty$, since $|T|^{\pm iy}$ is a unitary operator and $\phi(\frac{1}{2}+iy)=|T|^{-iy}BU|T||T|^{iy}$ and
$\omega_{C}(\cdot)$ is a weakly unitarily invariant seminorm on $M$, we have
$\omega_{C}(\phi(\frac{1}{2}+iy))=\omega_{C}(BU|T|).$ Note that
$$\phi(-\frac{1}{2}+iy)=|T|^{-iy}|T|BU|T|^{iy}=|T|^{-iy}U^{*}U|T|BU|T|^{iy},$$
by using the commutativity of $T$ and $B$, we have
$\omega_{C}(\phi(-\frac{1}{2}+iy))=\omega_{C}(BU|T|).$

Note that
\begin{align*}
\sup\limits_{U\in \mathscr{U}(\M)}F_{U}(-\frac{1}{2})&=\sup\limits_{U\in \mathscr{U}(\M)}Log \sup\limits_{ y\in \mathbb{R}}|f_{U}(-\frac{1}{2}+iy)|\\
&=Log \sup\limits_{ y\in \mathbb{R}}\sup\limits_{U\in \mathscr{U}(\M)}|f_{U}(-\frac{1}{2}+iy)|\\
&=Log \sup\limits_{ y\in \mathbb{R}}\sup\limits_{U\in \mathscr{U}(\M)}|\tau(CU\phi(-\frac{1}{2}+iy)U^{*})|\\
&=Log \sup\limits_{ y\in \mathbb{R}}\omega_{C}(\phi(-\frac{1}{2}+iy))\\
&=Log \omega_{C}(BU|T|).
\end{align*}

Similarly,
\begin{align*}
\sup\limits_{U\in \mathscr{U}(\M)}F_{U}(\frac{1}{2})&=\sup\limits_{U\in \mathscr{U}(\M)}Log \sup\limits_{ y\in \mathbb{R}}|f_{U}(\frac{1}{2}+iy)|\\
&=Log \sup\limits_{ y\in \mathbb{R}}\sup\limits_{U\in \mathscr{U}(\M)}|f_{U}(\frac{1}{2}+iy)|\\
&=Log \sup\limits_{ y\in \mathbb{R}}\sup\limits_{U\in \mathscr{U}(\M)}|\tau(CU\phi(\frac{1}{2}+iy)U^{*})|\\
&=Log \sup\limits_{ y\in \mathbb{R}}\omega_{C}(\phi(\frac{1}{2}+iy))\\
&=Log \omega_{C}(BU|T|).
\end{align*}

Then inequality (5.2) implies that for $-\frac{1}{2}\leq x \leq \frac{1}{2},$

\begin{align*}
\sup\limits_{U\in \mathscr{U}(\M)}F_{U}(x)&=\sup\limits_{U\in \mathscr{U}(\M)}Log \sup\limits_{ y\in \mathbb{R}}|f_{U}(x+iy)|\\
&=Log \sup\limits_{ y\in \mathbb{R}}\omega_{C}(\phi(x+iy))\\
&\leq Log\omega_{C}(BT),
\end{align*}
which means that $\omega_{C}(\phi(x+iy))\leq \omega_{C}(BT)$, $-\frac{1}{2}\leq x \leq \frac{1}{2},-\infty <y <\infty$, hence
$$\omega_{C}(|T|^{\lambda}BU|T|^{1-\lambda})\leq \omega_{C}(BT),\mbox{~~for~~}0\leq\lambda\leq1.$$
\end{proof}

The proof of the following proposition is exactly the same as \cite[Proposition 4]{Okubo2}, so we state it as follows without a proof.
\begin{proposition}\label{pro5}
Let $\M$ be a finite factor with a faithful normal trace $\tau$, $T\in \M$ be an invertible operator with polar decomposition $T=U|T|$. Let $\omega_{C}(\cdot)$ be the $C$-numerical radius  on $\M$ and $f(x)$ be a polynomial. Then
$$\omega_{C}(f(|T|^{\lambda}U|T|^{1-\lambda}))\leq \omega_{C}(f(T)),\mbox{~~for~~}0\leq\lambda\leq1.$$
\end{proposition}

Applying Theorem \ref{3} and Proposition \ref{pro5}, we can obtain that
\begin{proposition}\label{5.3}Let $\M$ be a finite factor with a faithful normal trace $\tau$. Assume $T\in \M$ is an invertible operator with polar decomposition $T=U|T|$ and $f$ is a polynomial, then for $0\leq\lambda\leq1$, $f(|T|^{\lambda}U|T|^{1-\lambda})$ is in the weak operator closure of the set $\{\sum_{i=1}^{n}z_{i}U_{i}f(T)U_{i}^{*}|~n\in \mathbb{N},(U_{i})_{1\leq i\leq n}\in \mathscr{U}(\M),\sum_{i=1}^{n}|z_{i}|\leq 1\}$.
 \end{proposition}

\proof[Acknowledgements]
The second author was supported by the Project sponsored by the NSFC grant 11431011 and startup funding from Hebei Normal University. The authors wish to express their thanks to Yongle Jiang for his carefully reading the draft of this paper and providing valuable suggestions
 and comments.

\bibliographystyle{amsplain}

\end{document}